\documentclass[11pt, leqno,twoside]{article}
\usepackage{amssymb}
\usepackage{amsmath}
\usepackage{amsthm}
\usepackage{amsfonts}
\usepackage{color}
\usepackage{textcomp}
\usepackage{graphicx}

\usepackage{reference}
\theoremstyle{definition}
\newtheorem*{acknowledgement}{Acknowledgement}

\allowdisplaybreaks

\def\Xint#1{\mathchoice
{\XXint\displaystyle\textstyle{#1}}%
{\XXint\textstyle\scriptstyle{#1}}%
{\XXint\scriptstyle\scriptscriptstyle{#1}}%
{\XXint\scriptscriptstyle\scriptscriptstyle{#1}}%
\!\int}
\def\XXint#1#2#3{{\setbox0=\hbox{$#1{#2#3}{\int}$ }
\vcenter{\hbox{$#2#3$ }}\kern-.6\wd0}}

\def\dashint{\Xint-}

\DeclareMathOperator{\divergence}{div}

\numberwithin{equation}{section}

\begin{document}
\title{ Liouville theorems for   stationary flows of shear thickening fluids in $2D$
\thanks
{E-mail addresses: \   guozhang@jyu.fi and  guozhang625@gmail.com}
}
\author{{ $\text{ Guo Zhang}$} \\
{\small Department of Mathematics and Statistics, P.O. Box 35 (MaD), FI-40014}\\
{\small University of Jyv\"{a}skyl\"{a}, Finland }}

\date{}
\maketitle

\begin{abstract}
In this paper we consider the entire weak solutions $u$ of the equations  for stationary flows of shear
thickening fluids in the plane and prove
Liouville theorems under the  conditions on the finiteness of energy and under the integrability condition of the solutions.
\end{abstract}

{\bf MR Subject Classification:}\ \ 76 D 05, 76 D 07, 76 M 30, 35 Q 30.
\bigbreak

\section{Introduction}
In this paper, we prove different types of Liouville theorems for
 the entire weak solutions $u:\mathbb{R}^2 \rightarrow \mathbb{R}^2,\  \pi:\mathbb{R}^2 \rightarrow \mathbb{R}$ of
 the following system

\begin{equation}\label{a}
\left\{
\renewcommand{\arraystretch}{1.25}
\begin{array}{lll}
- \divergence [T(\varepsilon(u))]+u^k\partial_{k}u+D\pi=0,\\
\divergence u=0\ \ \ \ \ \  \text{in\ \  $\mathbb{R}^2$},
\end{array}
\right.
\end{equation}
which describes the stationary flow of an incompressible generalized Newtonian fluid. In  equation \eqref{a},  $u$ denotes the velocity field, $\pi$
the pressure function,  $u^k\partial_{k}u$  the convective term, and $ T$ represents the stress tensor. As usual $\varepsilon(u)$ is the
symmetric derivative of $u$, i.e. $\varepsilon(u)=\frac{1}{2}(Du+(Du)^{T})=\frac{1}{2}(\partial_iu^k+\partial_ku^i)_{1\leq i,k \leq 2}$.

We assume that the stress tensor is the gradient of a potential $H:S^2\rightarrow \mathbb{R}$ defined on the space $S^2$ of all
symmetric $(2\times 2)$ matrices of the following form
\begin{equation}\label{b}
\renewcommand{\arraystretch}{1.25}
\begin{array}{lll}
H(\varepsilon)=h(|\varepsilon|),
\end{array}
\end{equation}
where $h$ is a nonnegative  function of class $C^2$. Thus
\begin{equation}\label{c}
 \renewcommand{\arraystretch}{1.25}
\begin{array}{lll}
T(\varepsilon)=DH(\varepsilon)=\mu(|\varepsilon|)\varepsilon, \ \ \ \mu(t)=\frac{h'(t)}{t}.
\end{array}
\end{equation}

Here $\mu$ denotes the viscosity coefficient. In case of generalized Newtonian fluids, it may depend on $\varepsilon(u)$.
This means that it depends on the motion of the fluids. If $\mu(t)$ is an increasing function,  the fluid is called shear thickening one.
If  $\mu(t)$ is a decreasing function,  the fluid is shear thinning. If $\mu(t)$ is a constant, then the fluid is Newtonian and 
$(1.1)$ reduces to the stationary Navier-Stokes equations for incompressible Newtonian fluids. For the further
mathematical and physical explanations, we refer to Ladyzhenskaya
\cite{La}, Galdi \cite{Ga1,Ga2}, Malek, Necas, Rokyta and Ruzicka
\cite{MNRR}, and Fuchs and Seregin \cite{FS}.

In the whole paper, we will concentrate on the following types of
shear thickening fluids. To be precise,  the potential $h$
satisfies the following conditions:
\[
\renewcommand{\arraystretch}{1.25}
\begin{array}{lll}
h\  \text{is strictly increasing and convex}\\
\text{together with $ h''(0) > 0$ and $\lim\limits_{t\rightarrow 0}\frac{h(t)}{t}=0$}.
\end{array}
\eqno(A1)
\]
\[
\renewcommand{\arraystretch}{1.25}
\begin{array}{lll}
\text{(doubling property) there exists a constant $a$} \\
\text{ such that} \  h(2t)\leq a h(t)\  \text{for all} \  t \geq 0.
\end{array}
\eqno(A2)
\]
\[
\renewcommand{\arraystretch}{1.25}
\begin{array}{lll}
\text{we have}\  \frac{h'(t)}{t} \leq h''(t)\  \text{for any}\  t \geq 0.
\end{array}
\eqno(A3)
\]

The study of Liouville type of theorems for Navier-Stokes equations goes back to the work of Gilbarg and Weinberger\cite{GW}. They
showed, among the others,  that the entire
solutions $u$ of stationary Navier-Stokes equations in the plane are constants under the condition:$\int_{\mathbb{R}^2}|Du|^2 dx$ $<\infty$.
For the  unstationary backward  Navier-Stokes equations in 2D, recently, Koch, Nadirashvili, Seregin and Sverak  \cite{KNSS} showed
that  $u(x,t)=b(t)\  \text{on}\  \mathbb{R}^2\times (-\infty,0)$ provided the solutions are bounded. Clearly, this result
implies the  Liouville theorem for stationary Navier-Stokes equations, that is, bounded solutions to stationary Navier-Stokes equations in $2D$ are constants.

For  general potential $h$ satisfying $(A1)-(A3)$, very
recently Fuchs \cite{Fu1}
showed  bounded solution $u$ of \eqref{a} must be a constant vector provided the solution satisfies the asymptotic
 behavior $|u-u_{\infty}|\rightarrow 0 $ at infinity, where $u_{\infty}$ is a constant vector. Later, the author removed
 the above assumption on $u$ at infinity and showed that
every bounded solution $u$ of \eqref{a} must be a constant vector in \cite{Z}.

Under  different hypothesists that the  flow is slow
(which means the convective term vanishes) and  energy is finite,
i.e. $\int_{\mathbb{R}^2}h(|Du|)dx<\infty$, Fuchs \cite{Fu1} showed
that the velocity field $u$ is a constant vector.
 In this note, we remove the assumption that the flow is slow and prove the following theorem which is the analogue result of Gilbarg and Weinberger in
the setting of shear thickening fluids.

\begin{theorem}\label{t1}
 Let $u\in C^1(\mathbb{R}^2,\mathbb{R}^2)$ be an entire weak
 solution to  \eqref{a} i.e.
\begin{equation}\label{d}
 \renewcommand{\arraystretch}{1.25}
\begin{array}{lll}
 \displaystyle \int_{\mathbb{R}^2}T(\varepsilon(u))\colon \varepsilon(\varphi) dx - \displaystyle \int_{\mathbb{R}^2}u^k u^i \partial_k \varphi^i dx = 0
\end{array}
\end{equation}
for all $\varphi \in C_0^\infty(\mathbb{R}^2,\mathbb{R}^2)$, $\divergence  \varphi=0$, and  satisfy the condition $\int_{\mathbb{R}^2}h(|Du|) dx<\infty$.
Then $u$ is a constant vector.
\end{theorem}

Next,  we consider another type of Liouville theorems for the
solutions of\eqref{a}. Recently, Fuchs \cite{Fu1} showed that the
solution is identically zero under the
conditions:\ \ $\int_{\mathbb{R}^2}h(|\varepsilon(u)|)dx<\infty$ and
$\int_{\mathbb{R}^2}|u|^2 dx<\infty$. Now we improve this result and
obtain the following types of Liouville theorems.

Before stating the results, let us introduce some notations. It  follows from $(A1)-(A3)$ that
there is a number $\tau \in (1, 2]$ such that

\[h'(t)\leq C(h(t)^{\frac{1}{\tau}}+1),\]
where $C>0$ is a constant, see Lemma 2.1 in section 2. We denote  $\tau'$ by its H\"{o}lder conjugate exponent, $\tau'=\frac{\tau}{\tau-1}$. Clearly, 
for the Navier-Stokes model, i.e., $h(t)=\frac{\nu}{2}t^2$,   $\tau=2$ and $\tau'=2$,  and for the Ladyzhenskaya model, i.e., $h(t)=\frac{\nu}{2}t^2+\mu t^p$,
 where $p>2$,  $\tau=\frac{p}{p-1}$ and $\tau'= p$.

\begin{theorem}\label{t2}
Suppose that the potential $h$ satisfies the conditions $(A1)-(A3)$. 
Let $u\in C^1(\mathbb{R}^2,\mathbb{R}^2)$ be  an entire weak solution to equation \eqref{a}. 
\begin{itemize}
\item[\rm{(i)}] Suppose that $3/2\le \tau\le 2$. Let $p$ be a number such that $p>1$. If $u\in L^p(\mathbb{R}^2,\mathbb{R}^2)$, then $u$ is the zero vector.
\item[\rm{(ii)}] Suppose that $4/3<\tau<3/2$. Let $p$ be a number such that $p>\tau^\prime$. If $u\in L^p(\mathbb{R}^2,\mathbb{R}^2)$, then $u$ is the zero vector.
\item[\rm{(iii)}] Suppose that $1<\tau\le 4/3$. If $u\in L^{\tau^\prime}(\mathbb{R}^2,\mathbb{R}^2)$, then $u$ is the zero vector.
\end{itemize}
\end{theorem}

In the setting of stationary Navier-Stokes equations,  we have the following corollary.

\begin{corollary}\label{c3}
 Let $u\in C^1(\mathbb{R}^2,\mathbb{R}^2)$ be an entire weak solution to stationary Navier-Stokes equations in the
plane, i.e.

\[
\renewcommand{\arraystretch}{1.25}
\begin{array}{lll}
 \displaystyle \int_{\mathbb{R}^2} Du\colon D\varphi dx - \displaystyle \int_{\mathbb{R}^2}u^k u^i \partial_k \varphi^i dx = 0
\end{array}
\]
for all $\varphi \in C_0^\infty(\mathbb{R}^2,\mathbb{R}^2)$, $\divergence \varphi=0$, and $u\in L^p(\mathbb{R}^2,\mathbb{R}^2),  p>1$. Then $u$ must be a zero vector.
\end{corollary}

We first  comment on  the regularity assumption on the solutions. As we known, for general $h$ satisfying $(A1)-(A3)$,
$C^{1,\alpha}$ regularity of solutions is an open problem. But in some special case, such as $h(t)=t^2(1+t)^m, m\geq0$, the solution $u$ belongs
 to the space $C^{1,\alpha}$ \cite{BFZ}. For the further discussion about regularity of the solutions of \eqref{a} the readers are referred to
 see the paper \cite{Fu2}.  Therefore the regularity assumptions on $u$ in  the above theorems are reasonable.

Second, we comment about the proofs of the theorems. For obtaining Theorem  \pref{t1}, in the special case $h(t)=\frac{\nu}{2}t^2$, Gilbarg
 and Weinberger's \cite{GW} approach relies on the following fact: the vorticity  $\omega=\partial_{x_1}u^2-\partial_{x_2}u^1$ satisfies the elliptic 
equation $-\triangle \omega+u\cdot D\omega=0 $ and hence it satisfies the maximum principle. In our setting for equation\eqref{a}, it seems that this
approach does not work. We follow the approach of Fuchs in \cite{Fu1}, see also \cite{FZ} and \cite{Z}. The essential idea is to study the energy 
estimate for the second order derivatives, see Lemma\ref{L3.1} in Section $3$. Then we conclude that $u$ must be a constant vector under
the condition $\int_{\mathbb{R}^2}h(|Du|)dx < \infty$.

The idea for proving \tref{t2} is as follows:  if $p$ is ``suitable'' small,
we directly use the local energy estimate for the first order derivatives  to control the local integral of $|u|^q$, where $q$ is large enough,
and conclude that $\int_{\mathbb{R}^2}|u|^qdx=0$.  If $p$ is ``suitable'' large, we prove the local uniformly finite energy estimate for the second order derivatives , 
from which follows the boundedness of solutions. Then we conclude the proof of \tref{t2} by the Liouville theorem for bounded solutions in  \cite{Z}.

Our notations is standard. Throughout this paper,  the convention of summation with respect to indices repeated twice is used. All constants
 are denoted by the symbol $ C,$ and  $ C $ may change from line to line, whenever it is necessary we will indicate the dependence of
$C$ on parameters. As usual $Q_R(x_0)$ denotes the open square with center $x_0$ and side length $2R,$  and symbols $\colon,$  $\cdotp$ will be
used for the scalar product of matrices and vectors respectively. $|\cdotp|$ denotes the associated Eculidean norms.

Our paper is organized as follows: In section $2$,  we present some
auxiliary results. In section $3$,  we give the proof of Theorem\pref{t1}, and 
in section $4$,  we give the proof of Theorem\pref{t2}.


\section{Auxiliary Results}
\subsection{The properties of function $h$}
The following properties of function $h$ follow from $(A1)-(A3)$, see \cite{Fu1}.

$(i)$\ \ $\mu(t)=\frac{h'(t)}{t}$ is an increasing function.

$(ii)$ \ \ We have $h(0)=h'(0)=0$ and

\begin{equation}\label{1}
 \renewcommand{\arraystretch}{1.25}
\begin{array}{lll}
h(t) \geq \frac{1}{2} h''(0)t^2.
\end{array}
\end{equation}

Moreover,

\begin{equation}\label{2}
\renewcommand{\arraystretch}{1.25}
\begin{array}{lll}
\frac{ h'(t)}{t} \geq \lim\limits_{s\rightarrow 0}\frac{h'(s)}{s}=h''(0)>0.
\end{array}
\end{equation}

$(iii)$\ \ It satisfies the balancing condition, i.e., for some $a>0$, 

\begin{equation}\label{3}
\renewcommand{\arraystretch}{1.25}
\begin{array}{lll}
\frac{1}{a}h'(t)t\leq h(t) \leq th'(t),\ \ \ \ t\geq 0.
\end{array}
\end{equation}

$(iv)$\ \ For an exponent $m \geq 2$ and a constant $C \geq 0$ it holds

\begin{equation}\label{4}
\renewcommand{\arraystretch}{1.25}
\begin{array}{lll}
h(t)\leq C (1+t^m),\ \ \ \ h'(t)\leq C(1+t^m),\ \ \ \ t \geq 0.
\end{array}
\end{equation}

From the assumptions on $h,$ we know the system satisfies the following elliptic condition, $\forall \varepsilon, \sigma \in S^2,$

\begin{equation}\label{5}
\renewcommand{\arraystretch}{1.25}
\begin{array}{lll}
\frac{h'(|\varepsilon|)}{|\varepsilon|}|\sigma|^2 \leq D^2H(\varepsilon)(\sigma,\sigma)\leq h''(|\varepsilon|)|\sigma|^2,
\end{array}
\end{equation}
from which, together with\eqref{2}, it follows that

\begin{equation}\label{6}
\renewcommand{\arraystretch}{1.25}
\begin{array}{lll}
 D^2H(\varepsilon)(\sigma,\sigma)\geq h''(0)|\sigma|^2.
\end{array}
\end{equation}

Finally, we have the following lemma which is taken from \cite{Fu1}.

\begin{lemma}\label{L2.4}
There is a number $\tau \in (1, 2]$ such that
\[
 h'(t)\leq C(h(t)^{\frac{1}{\tau}}+1)
\]
or equivalently

\[
 |DH(\varepsilon(u))|\leq C(H(\varepsilon(u))^{\frac{1}{\tau}}+1)
\]
holds for all $t \geq 0$ and $\varepsilon \in S^2$. Moreover,  we 
have the sharper estimate

\[
 h'(t)\leq C(h(t)^{\frac{1}{\tau}}+t),\ t\geq 0.
\]
\end{lemma}

\subsection{Divergence equations and Korn's inequality}

First, we introduce a standard result concerning the ``divergence equations'', see e.g. \cite{Ga1}, \cite{Ga2} or \cite{FS}. For any $R>0$ and $x_0\in\mathbb{R}^2$,
define $ Q_R(z)=\{
(\tilde{x},\tilde{y})\in\mathbb{R}^2 \mid |\tilde{x}-x|<R, |\tilde{y}-y|<R, z=(x,y)$\}.

\begin{lemma}\label{L2.1}
Consider a function $f \in L^q(Q_R(z)), q>1$ such that $\int_{Q_R(z)}f dx=0$. Then there exists a field $
v\in W_0^{1,q}(Q_R(z),\mathbb{R}^2)$ and a constant $C(q)$,  independent of $Q_R(z)$,  such that we have $\divergence v =f$ on $Q_R(z)$ together
 with the estimate
\[
\displaystyle\int_{Q_R (z)} |Dv|^q dx \le C(q) \int_{Q_R (z)} |f|^q dx \, .
\]
\end{lemma}

Second, the following lemma is  the classical Korn inequality, see \cite{Te}.

\begin{lemma}\label{L2.2}
There is an absolute constant $C$ such that for all $v\in W_0^{1,2}(Q_R(z),\mathbb{R}^2)$ it holds
\[
\int_{Q_R (z)} |Dv|^2 dx \le C \int_{Q_R (z)} |\varepsilon(v)|^2 dx \,,
\]
and for all $v\in W^{1,2}(Q_R (z),\mathbb{R}^2)$ it holds

\[
\int_{Q_R (z)} |Dv|^2 dx \leq C \biggl( \int_{Q_R (z)} |\varepsilon(v)|^2 dx+\dfrac{1}{R^2}\int_{Q_R (z)}|v|^2 dx \biggr).
\]
\end{lemma}
\subsection{Ladyzhenskaya's inequality and Sobolev-Poincar\'{e}'s inequality}
First, we introduce a local version of the Sobolev inequality in the plane, see \cite{GT}.
\begin{lemma}\label{L2.8}
Let $x_0\in \mathbb{R}^2$, $R>0$, $Q_R(x_0)\subset \mathbb{R}^2$ and $u\in W^{1,2}(Q_R(x_0))$. Then, $\forall q>1$, there exists a constant $C(q)$ depending
only on $q$ such that the following inequality holds
\[
\renewcommand{\arraystretch}{1.25}
\begin{array}{lll}
\biggl(\dfrac{1}{R^2}\displaystyle\int_{Q_R(x_0)}|u|^q dx\biggr)^{\frac{1}{q}}\leq C(q) \biggl\{\biggl(\displaystyle\int_{Q_R(x_0)}|Du|^2dx \biggr)^{\frac{1}{2}} +\biggl(\dfrac{1}{R^2}\displaystyle\int_{Q_R(x_0)}|u|^2 dx\biggr)^{\frac{1}{2}} \biggr\}.
\end{array}
\]
\end{lemma}

Next, we need a local version of Ladyzhenskaya's inequality. It is an easy consequence of Ladyzhenskaya's inequality. We give the proof here.

\begin{lemma}\label{L2.5}
Suppose $u\in W^{1,2}(Q_R(x_0))$, $Q_R(x_0)\subset \mathbb{R}^2 $. Then there exists a constant $C_0$ independent of $R$, $x_0$ such that

\[
\renewcommand{\arraystretch}{1.25}
\begin{array}{lll}
 \displaystyle\int_{Q_R(x_0)} |u|^4 dx\leq C_0 \biggl\{\displaystyle\int_{Q_R(x_0)} |u|^2 dx\displaystyle\int_{Q_R(x_0)} |Du|^2 dx
 +\frac{1}{R^2}\biggl(\displaystyle\int_{Q_R(x_0)} |u|^2 dx\biggr)^2 \biggr\}.
\end{array}
\]
\end{lemma}

\begin{proof}
For any $x_0\in \mathbb{R}^2$ and  $v\in W^{1,2}(Q_1(x_0))$, then, there exists an extension $\tilde v \in W_0^{1,2}(Q_2(x_0))$ of $v$ s.t.

\begin{equation}\label{5.1}
\renewcommand{\arraystretch}{1.25}
\begin{array}{lll}
\|\tilde v\|_{L^2(Q_2(x_0))}\leq C \|v\|_{L^2(Q_1(x_0))}
\end{array}
\end{equation}
and
\begin{equation}\label{5.2}
\renewcommand{\arraystretch}{1.25}
\begin{array}{lll}
\|\tilde v\|_{W_0^{1,2}(Q_2(x_0))}\leq C \|v\|_{W^{1,2}(Q_1(x_0))},
\end{array}
\end{equation}
where $C$ is an absolute constant. See \cite{Ev}.

Moreover, by Ladyzhenskaya's inequality (see \cite{Te1}) we have
\begin{equation}\label{5.3}
\renewcommand{\arraystretch}{1.25}
\begin{array}{lll}
\displaystyle\int_{Q_2(x_0)}|\tilde v|^4 dx \leq 2\displaystyle\int_{Q_2(x_0)}|\tilde v|^2dx \displaystyle\int_{Q_2(x_0)}|D\tilde v|^2dx .
\end{array}
\end{equation}

Combing the estimates \eqref{5.1}, \eqref{5.2} and \eqref{5.3},  we obtain that
\[
\renewcommand{\arraystretch}{1.25}
\begin{array}{lll}
\displaystyle\int_{Q_2(x_0)}|\tilde v|^4 dx &\leq  C \biggl(\displaystyle\int_{Q_1(x_0)}|D v|^2dx +\displaystyle\int_{Q_1(x_0)}| v|^2dx\biggr)\displaystyle\int_{Q_1(x_0)}| v|^2dx\\
&\leq  C\biggl\{\displaystyle\int_{Q_1(x_0)}| v|^2dx \displaystyle\int_{Q_1(x_0)}|D v|^2dx  +\biggl(\displaystyle\int_{Q_1(x_0)}| v|^2dx\biggr)^2 \biggr\},
\end{array}
\]
from which, it follows that
\begin{equation}\label{5.4}
\renewcommand{\arraystretch}{1.25}
\begin{array}{lll}
\displaystyle\int_{Q_1(x_0)}| v|^4 dx
\leq  C\biggl\{ \displaystyle\int_{Q_1(x_0)}| v|^2dx\displaystyle\int_{Q_1(x_0)}|D v|^2dx  +\biggl(\displaystyle\int_{Q_1(x_0)}| v|^2dx\biggr)^2 \biggr\}.
\end{array}
\end{equation}

For  $R>0$, let $v(x):=u(Rx)$, thus, $v(x)\in W^{1,2}( Q_1(x_0))$.   Thus \eqref{5.4} holds, and hence,  we end up with 
\[
\renewcommand{\arraystretch}{1.25}
\begin{array}{lll}
\displaystyle\int_{Q_R(x_0)}| u|^4 dx
\leq  C\biggl\{ \displaystyle\int_{Q_R(x_0)}| u|^2dx\displaystyle\int_{Q_R(x_0)}|D u|^2dx  +\dfrac{1}{R^2}\biggl(\displaystyle\int_{Q_R(x_0)}| u|^2dx\biggr)^2 \biggr\}.
\end{array}
\]
This finishes the proof.
\end{proof}
\subsection{A lemma of Gilbarg and Weinberger}

The following result was due to Gilbarg and Weinberger, see Lemma $2.1$ in \cite{GW}.

\begin{lemma}\label{L2.3}
Let $f\in C^1(\mathbb{R}^2,\mathbb{R}^2)$ in $r>r_0>0$ and have finite integral
\[
\displaystyle\int_{r>r_0} |Df|^2 dxdy <\infty.
\]

Then
\[
 \displaystyle\lim_{r\rightarrow \infty}\frac{1}{\text{log} r} \int_0^{2\pi} f(r,\theta)^2 d\theta =0.
\]
\end{lemma}

\subsection{A lemma of  Giaquinta and Modica }

The following $\varepsilon$-lemma goes back to the work of Giaquinta and Modica \cite{GM}. Recently,  a generalized version of $\varepsilon$-lemma 
was given in  \cite{FZ}. For proving our results, we need the following version of $\varepsilon$-lemma.

\begin{lemma}\label{L2.6}
Let $f,$ $f_1, \ldots,f_l$ denote non-negative functions from the space $L^1_{loc} (\mathbb{R}^2).$ Suppose further
that we are given exponents $\alpha_1, \ldots, \alpha_l \geq 0,$  $\beta_1, \ldots, \beta_l \geq 1.$ For any $x_0\in \mathbb{R}^2$, $Q=Q_{2R}(x_0)$,
we can find $\delta_0$ depending on $\alpha_1, \ldots, \alpha_l \geq 0 $  as follows: if for $\delta \in (0, \delta_0)$ it is possible to calculate a 
constant $C (\delta) > 0$ such that the inequality

\begin{equation}\label{2.11}
\renewcommand{\arraystretch}{1.25}
\begin{array}{lll}
\displaystyle\int_{Q_r (z)} f dx \le \delta \int_{Q_{2r} (z)} f dx + C (\delta) \sum^l_{j = 1} r^{-\alpha_j} \biggl(\int_{Q_{2r}(z)} f_j dx\biggr)^{\beta_j}
\end{array}
\end{equation}
holds for any choice $Q_{2r}(z)\subset Q_{2R}(x_0)$. Then there is a
constant $C$ independent of $\delta$ and $R$ with the property
\[
\int_{Q_R (x_0)} f dx \leq  C \sum^l_{j = 1} R^{-\alpha_j} \biggl(\int_{Q_{2R}(x_0)} f_j dx\biggr)^{\beta_j}.
\]
\end{lemma}

\begin{remark}
When $\beta_j=1,  j=1, 2,\ldots,l$, Lemma\pref{L2.6} is reduced to Lemma $3.1$ of \cite{FZ}. Notice we have the trivial inequality 
$(\int_{Q_{2r}(z)} f_j dx)^{\beta_j}\leq (\int_{Q_{2R}(x_0)} f_j dx)^{\beta_j-1}\int_{Q_{2r}(z)} f_j dx$. In this way, we can reduce the 
assumption\eqref{2.11} to that of Lemma $3.1$ of \cite{FZ}. Then the proof of Lemma\pref{L2.6} is exactly the same as that of Lemma $3.1$ of \cite{FZ}.
\end{remark}

\subsection{A Liouville theorem}
We need the following Liouville theorem, Theorem $1$ in \cite{Z}, for the equation \eqref{a},
to prove Theorem\pref{t2}. 

\begin{theorem}\label{L2.7}
Suppose $u\in C^1(\mathbb{R}^2,\mathbb{R}^2) \bigcap L^{\infty}(\mathbb{R}^2,\mathbb{R}^2)$ be an entire
 weak solution of\eqref{a}. Then $u$ is a constant vector.
\end{theorem}

\section{ Proof of Theorem\ref{t1}}

 In  view of $u\in C^1(\mathbb{R}^2,\mathbb{R}^2)$ and the elliptic condition\eqref{5}, by standard difference quotient technique we
 can prove  that $u\in W_{loc}^{2,2}(\mathbb{R}^2,\mathbb{R}^2)$. See \cite{Fu1}, \cite{Z}. The Proof of
 Theorem\ref{t1} is divided into the following three lemmas.

\begin{lemma}\label{L3.1}
Let $u\in C^1(\mathbb{R}^2,\mathbb{R}^2)$ be an entire weak solution of\eqref{a} and satisfy the
  condition $\int_{\mathbb{R}^2}h(|Du|) dx<\infty$. Then, for any $x_0\in \mathbb{R}^2$, $R>0,$ the following energy estimate holds
\begin{equation}\label{3.1}
\renewcommand{\arraystretch}{1.25}
\begin{array}{lll}
 \displaystyle\int_{Q_R(x_0)} W dx\leq & C \biggl\{\displaystyle\frac{1}{R^2}\int_{Q_{2R}(x_0)} h(|\varepsilon(u)|)dx
 +\displaystyle\frac{1}{R^2}\int_{Q_{2R}(x_0)} |Du|^2 dx\\
&+\biggl(1+\displaystyle\frac{1}{R^{2m}} \biggr)
+\displaystyle\frac{1}{R^3}\int_{Q_{2R}(x_0)} |u|dx \biggr\},
\end{array}
\end{equation}
where $W= D^2H(\varepsilon(u))  (\varepsilon(\partial_ku), \varepsilon(\partial_ku))$,\  $m>0$, $C$ is a constant independent of $x_0$, $R$.
\end{lemma}

\begin{proof}
For any cut-off function $\eta\in C_0^{\infty}(\mathbb{R}^2)$, $0\leq\eta\leq 1$,  the following estimate is obtained in \cite{Z}, see $(3.9)$ of
\cite{Z},

\begin{equation}\label{3.2}
 \renewcommand{\arraystretch}{1.25}
\begin{array}{lll}
 \displaystyle\int_{\mathbb{R}^2} & D^2H(\varepsilon(u))  (\varepsilon(\partial_ku), \varepsilon(\partial_ku))\eta^2dx\\
&\leq C\biggl\{ \displaystyle\int_{\mathbb{R}^2} h(|\varepsilon(u)|)|D\eta|^2 dx+
  \displaystyle\int_{\mathbb{R}^2} h'(|\varepsilon(u)|)^2(|D\eta|^2+|D^2\eta|) dx\\
&+ \displaystyle\int_{\mathbb{R}^2}|Du|^2(|D\eta|^2+|D^2\eta|)dx + \displaystyle\int_{\mathbb{R}^2}|Du|^2|u||D\eta|dx \biggr\}.
\end{array}
\end{equation}

Now, for any $x\in Q_{2R}(x_0), r>0$, $Q_{2r}(x)\subset Q_{2R}(x_0)$ and  $\eta\in C_{0}^{\infty}(Q_{\frac{3}{2}r}(x))$ satisfying
$\eta=1$ in $Q_r(x)$ and $0\leq\eta\leq 1$, $|D\eta|\leq \frac{4}{r}$, $|D^2\eta|\leq \frac{16}{r^2}$, we deduce from\eqref{3.2} that
\begin{equation}\label{3.3}
\renewcommand{\arraystretch}{1.25}
\begin{array}{lll}
 \displaystyle\int_{Q_r(x)} W dx
&\leq C\biggl\{ \displaystyle\frac{1}{r^2}\int_{Q_{\frac{3}{2}r}(x)} h(|\varepsilon(u)|)dx+
  \displaystyle\frac{1}{r^2}\int_{Q_{\frac{3}{2}r}(x)} h'(|\varepsilon(u)|)^2 dx\\
&+ \displaystyle\frac{1}{r^2}\int_{Q_{\frac{3}{2}r}(x)}|Du|^2dx + \displaystyle\frac{1}{r}\int_{T_{\frac{3}{2}r}(x)}|Du|^2|u|dx \biggr\},
\end{array}
\end{equation}
where $T_{\frac{3}{2}r}(x)=Q_{\frac{3}{2}r}(x)\setminus \overline{Q_r(x)}$.

For the term $\frac{1}{r^2}\int_{Q_{\frac{3}{2}r}(x) }h'(|\varepsilon(u)|)^2 dx$,  we have the following estimate,for $L>0$,
see $(3.16)$ in \cite{Z},
\begin{equation*}\label{3.4}
\renewcommand{\arraystretch}{1.25}
\begin{array}{lll}
\displaystyle \frac{1}{r^2}\int_{Q_{\frac{3}{2}r}(x)}h'(|\varepsilon(u)|)^2 dx&\leq \displaystyle C  h'(L)^2+C\frac{1}{L^2}\frac{1}{r^4}
 \biggl(\int_{Q_{2r}(x)}h(|\varepsilon(u)|) dx\biggr)^2\\
&+\displaystyle C \frac{1}{r^2}\frac{1}{L^2} \int_{Q_{2r}(x)}h(|\varepsilon(u)|) dx \int_{Q_{2r}(x)}W dx.
\end{array}
\end{equation*}

Choosing $L=\frac{1}{\varepsilon^{\frac{1}{2}}r}$, $\varepsilon<1$, we have that
\begin{equation}\label{3.5}
\renewcommand{\arraystretch}{1.25}
\begin{array}{lll}
\displaystyle \frac{1}{r^2}\int_{Q_{\frac{3}{2}r}(x)}h'(|\varepsilon(u)|)^2 dx&\leq \displaystyle C
 h'(\frac{1}{\varepsilon^{\frac{1}{2}}r})^2+C\frac{1}{r^2}  \biggl(\int_{Q_{2r}(x)}h(|\varepsilon(u)|) dx\biggr)^2\\
&+\displaystyle C \varepsilon \int_{Q_{2r}(x)}h(|\varepsilon(u)|) dx \int_{Q_{2r}(x)}Wdx.
\end{array}
\end{equation}

We then deal with the last term  in\eqref{3.3}. 
Letting $ A=\dashint_{T_{\frac{3}{2}r}(x)}|Du|^2 dx $ and $B=\dashint_{T_{\frac{3}{2}r}(x)}u dx$, we have
\begin{equation*}\label{3.6}
\renewcommand{\arraystretch}{1.25}
\begin{array}{lll}
\displaystyle \frac{1}{r}\int_{T_{\frac{3}{2}r}(x)}|Du|^2|u|dx \leq &\displaystyle  \frac{1}{r}\int_{T_{\frac{3}{2}r}(x)}||Du|^2-A||u-B|dx\\
&+\displaystyle\frac{1}{r}|B|\int_{T_{\frac{3}{2}r}(x)}|Du|^2dx+\displaystyle\frac{1}{r}|A|\int_{T_{\frac{3}{2}r}(x)}|u-B|dx.
\end{array}
\end{equation*}

Recalling the choices of $A$ and $B$, we obtain by Young's inequality, for $\varepsilon>0$,
\begin{equation*}\label{3.7}
\renewcommand{\arraystretch}{1.25}
\begin{array}{lll}
\displaystyle \frac{1}{r}\int_{T_{\frac{3}{2}r}(x)}|Du|^2|u|dx \leq &\displaystyle  \varepsilon\int_{T_{\frac{3}{2}r}(x)}||Du|^2-A|^2dx
+\displaystyle\frac{1}{\varepsilon}\frac{1}{r^2}\int_{T_{\frac{3}{2}r}(x)}|u-B|^2dx\\
&+\displaystyle\frac{C}{r^3}\int_{T_{\frac{3}{2}r}(x)}|Du|^2dx\int_{T_{\frac{3}{2}r}(x)}|u|dx.
\end{array}
\end{equation*}

Applying Poincar\'{e}'s inequality and Sobolev-Poincar\'{e}'s inequality  we obtain  that
\begin{equation*}\label{3.8}
\renewcommand{\arraystretch}{1.25}
\begin{array}{lll}
\displaystyle \frac{1}{r}\int_{T_{\frac{3}{2}r}(x)}|Du|^2|u|dx \leq &\displaystyle  \varepsilon\biggl(\int_{T_{\frac{3}{2}r}(x)}|D(|Du|^2)|dx\biggr)^2
+\displaystyle\frac{1}{\varepsilon } \int_{T_{\frac{3}{2}r}(x)}|Du|^2dx\\
&+\displaystyle\frac{C}{r^3}\int_{T_{\frac{3}{2}r}(x)}|Du|^2dx\int_{T_{\frac{3}{2}r}(x)}|u|dx,
\end{array}
\end{equation*}
from which, together with H\"{o}lder's inequality, it follows that
\begin{equation}\label{3.9}
\renewcommand{\arraystretch}{1.25}
\begin{array}{lll}
\displaystyle \frac{1}{r}\int_{T_{\frac{3}{2}r}(x)}|Du|^2|u|dx \leq & \displaystyle  \varepsilon\int_{T_{\frac{3}{2}r}(x)}|Du|^2dx
\int_{T_{\frac{3}{2}r}(x)}|D^2u|^2dx\\
&+\displaystyle\frac{1}{\varepsilon } \int_{T_{\frac{3}{2}r}(x)}|Du|^2dx
+\displaystyle\frac{C}{r^3}\int_{T_{\frac{3}{2}r}(x)}|Du|^2dx\int_{T_{\frac{3}{2}r}(x)}|u|dx.
\end{array}
\end{equation}

Combining\eqref{3.3},\eqref{3.5} and\eqref{3.9} and observing the inequality $|D^2u(x)|\leq C|D\varepsilon(u)(x)|\leq C W(x)$ we deduce that
\begin{equation}\label{3.10}
\renewcommand{\arraystretch}{1.25}
\begin{array}{lll}
\displaystyle \int_{Q_r(x)} Wdx\leq  & C \varepsilon \displaystyle\int_{T_{\frac{3}{2}r}(x)}|Du|^2dx \displaystyle \int_{Q_{2r}(x)} Wdx\\
&+C \varepsilon \displaystyle \int_{Q_{2r}(x)} h(|\varepsilon(u)|)dx  \int_{Q_{2r}(x)} Wdx\\
&+\displaystyle\frac{C}{r^2}\int_{Q_{2r}(x)}
h(|\varepsilon(u)|)dx+\displaystyle\frac{C}{
 r^2}\int_{Q_{2r}(x)} |Du|^2dx\\
&+C h'(\frac{1}{\varepsilon^{\frac{1}{2}}r})^2 
+\displaystyle\frac{C}{r^2} (\displaystyle \int_{Q_{2r}(x)} h(|\varepsilon(u)|)dx)^2\\
&+\displaystyle \frac{C}{\varepsilon}  \int_{T_{\frac{3}{2}r}(x)}|Du|^2dx+
\displaystyle\frac{C}{r^3}\int_{T_{\frac{3}{2}r}(x)}|Du|^2dx\int_{T_{\frac{3}{2}r}(x)}|u|dx.
\end{array}
\end{equation}

Since $\int_{\mathbb{R}^2}|Du|^2dx\leq C\int_{\mathbb{R}^2}h(|Du|)dx < \infty$, choosing $\varepsilon$ small enough and 
denoting $\delta:=C \varepsilon<\frac{1}{2}$,  we obtain from\eqref{3.10} that
\begin{equation}\label{3.11}
\renewcommand{\arraystretch}{1.25}
\begin{array}{lll}
\displaystyle \int_{Q_r(x)} Wdx\leq  &  \displaystyle \delta\int_{Q_{2r}(x)} Wdx
+\displaystyle\frac{C}{r^2}\int_{Q_{2r}(x)} h(|\varepsilon(u)|)dx+\displaystyle\frac{C}{r^2}\int_{Q_{2r}(x)} |Du|^2dx \\
&+C\displaystyle\int_{T_{\frac{3}{2}r}(x)}|Du|^2dx+\displaystyle C h'(\frac{1}{\varepsilon^{\frac{1}{2}}r})^2 +
\displaystyle\frac{C}{r^3}\int_{Q_{\frac{3}{2}r}(x)}|u|dx.
\end{array}
\end{equation}

In view of the condition $h'(t)\leq C(1+t^m)$ and $\int_{\mathbb{R}^2}|Du|^2dx\leq C \int_{\mathbb{R}^2}h(|Du|)dx < \infty$ it follows that

\begin{equation*}\label{3.12}
\renewcommand{\arraystretch}{1.25}
\begin{array}{lll}
\displaystyle \int_{Q_r(x)} Wdx\leq  &  \displaystyle \delta\int_{Q_{2r}(x)} Wdx
+\displaystyle\frac{C}{r^2}\int_{Q_{2r}(x)} h(|\varepsilon(u)|)dx+\displaystyle\frac{C}{r^2}\int_{Q_{2r}(x)} |Du|^2dx \\
&+\displaystyle C (1+\frac{1}{r^{2m}})+
\displaystyle\frac{C}{r^3}\int_{Q_{2r}(x)}|u|dx.
\end{array}
\end{equation*}

By Lemma\ref{L2.6} we end up with
\begin{equation*}\label{3.13}
\renewcommand{\arraystretch}{1.25}
\begin{array}{lll}
 \displaystyle\int_{Q_R(x_0)} W dx\leq & C \biggl\{\displaystyle\frac{1}{R^2}\int_{Q_{2R}(x_0)} h(|\varepsilon(u)|)dx
 +\displaystyle\frac{1}{R^2}\int_{Q_{2R}(x_0)} |Du|^2 dx\\
&+\biggl(1+\displaystyle\frac{1}{R^{2m}} \biggr)
+\displaystyle\frac{1}{R^3}\int_{Q_{2R}(x_0)} |u|dx \biggr\}.
\end{array}
\end{equation*}
\end{proof}

\begin{lemma}\label{L3.2}
Let $u$ be as in Lemma\ref{L3.1}. Then the following estimate holds
\begin{equation}\label{3.14}
\renewcommand{\arraystretch}{1.25}
\begin{array}{lll}
 \displaystyle\int_{\mathbb{R}^2}D^2H(\varepsilon(u))(\varepsilon(\partial_ku), \varepsilon(\partial_ku))dx< \infty.
\end{array}
\end{equation}

Therefore,
\begin{equation}\label{3.15}
\renewcommand{\arraystretch}{1.25}
\begin{array}{lll}
 \displaystyle\int_{\mathbb{R}^2}|D^2u|^2dx< \infty.
\end{array}
\end{equation}
\end{lemma}

\begin{proof}
Since $\int_{\mathbb{R}^2}h(|Du|)dx=:M<\infty$, for $R>1$,\eqref{3.1} gives
\begin{equation}\label{3.17}
\renewcommand{\arraystretch}{1.25}
\begin{array}{lll}
 \displaystyle\int_{Q_R(x_0)} W dx\leq C(M)+ \dfrac{C}{R^3} \displaystyle\int_{Q_{2R}(x_0)} |u|dx.
\end{array}
\end{equation}

Since $u\in C^1(\mathbb{R}^2, \mathbb{R}^2)$ and
$\int_{\mathbb{R}^2}|Du|^2dx\leq C\int_{\mathbb{R}^2}h(|Du|)dx<\infty$,
using Lemma\ref{L2.3} we deduce that

\begin{equation}\label{3.18}
\renewcommand{\arraystretch}{1.25}
\begin{array}{lll}
\limsup\limits_{R\rightarrow\infty}\dfrac{1}{R^3} \displaystyle\int_{Q_{2R}(x_0)} |u|dx=0.
\end{array}
\end{equation}

Letting $R\rightarrow\infty$ in\eqref{3.17} we have
\begin{equation}\label{3.19}
\renewcommand{\arraystretch}{1.25}
\begin{array}{lll}
 \displaystyle\int_{\mathbb{R}^2} W dx< \infty.
\end{array}
\end{equation}

Since $|D^2u(x)|\leq C |D\varepsilon(u)(x)|$, then\eqref{3.19} implies\eqref{3.15}.  The proof is complete.
\end{proof}

\begin{lemma}\label{L3.3}
Let $u$ be as in Lemma\ref{L3.1}, then  we have
\begin{equation}\label{3.20}
\renewcommand{\arraystretch}{1.25}
\begin{array}{lll}
 \displaystyle\int_{\mathbb{R}^2}Wdx=0.
\end{array}
\end{equation}
Hence, $u$ must be a constant vector.
\end{lemma}

\begin{proof}
Since $\int_{\mathbb{R}^2}|Du|^2dx\leq C\int_{\mathbb{R}^2}h(|Du|)dx<\infty$, it implies 
that $\lim\limits_{r\rightarrow\infty} \int_{T_{\frac{3}{2}r}(x)}|Du|^2dx$ $=0$. Letting $r\rightarrow \infty$ in\eqref{3.11}, we have by 
the condition $h'(0)=0$,\eqref{3.18} and\eqref{3.19} that

\begin{equation}\label{3.21}
\renewcommand{\arraystretch}{1.25}
\begin{array}{lll}
 \displaystyle\int_{\mathbb{R}^2}Wdx \leq \dfrac{1}{2} \displaystyle\int_{\mathbb{R}^2}Wdx.
\end{array}
\end{equation}
Thus\eqref{3.20} holds, and hence $W(x)= 0$. From the relation $|D^2u(x)|^2\leq C |D\varepsilon(u)(x)|^2\leq C W(x)$, we know that $D^2u(x)= 0$. Therefore  $u$
must be  an affine function. On the other hand, in view of inequality $\int_{\mathbb{R}^2}|Du|^2dx\leq C\int_{\mathbb{R}^2}h(|Du|)dx<\infty$, it gives $Du(x)= 0$. Then, $u$ is a constant vector.
\end{proof}

\section{Proof of \ref{t2}}

To prove\tref{t2}, the Liouville theorem under the integrability condition of $u$, we need the energy estimates for the first 
order derivatives and the second order derivatives. The following Lemma gives that for the first order derivatives.

\begin{lemma}\label{L4.1}
Let $u\in C^1(\mathbb{R}^2, \mathbb{R}^2)$ be an entire weak solution of\eqref{a}. Then, for any $x_0\in \mathbb{R}^2$,
$R>0$, the following energy estimate holds

\begin{equation}\label{4.1}
\renewcommand{\arraystretch}{1.25}
\begin{array}{lll}
 \displaystyle\int_{Q_R(x_0)}h(|\varepsilon(u)|)dx \leq  &C\biggl\{\dfrac{1}{R^{\tau'}} \displaystyle\int_{Q_{2R}(x_0)}|u|^{\tau'} dx
 +\dfrac{1}{R^2}\displaystyle\int_{Q_{2R}(x_0)}|u|^2 dx \\
&+\dfrac{1}{R^2}  \biggl(\displaystyle\int_{Q_{2R}(x_0)}|u|^2 dx \biggr)^2  \biggr\},
\end{array}
\end{equation}
where $\tau'=\frac{\tau}{\tau-1}$ and $\tau, 1<\tau\leq 2$, is as in Lemma\ref{L2.4}.
\end{lemma}

\begin{proof}
 For any $x_0\in \mathbb{R}^2$,  $R>0$ and $x\in Q_{2R}(x_0)$, $r>0$ s.t.  $Q_{2r}(x)\subset Q_{2R}(x_0)$, we
choose the cut-off function $\eta\in C_0^{\infty}(Q_{\frac{3}{2}r}(x))$ as Lemma\ref{L3.1} and find a solution $\varpi$ to the following equation
\begin{equation*}\label{4.2}
\renewcommand{\arraystretch}{1.25}
\begin{array}{lll}
 \divergence\varpi= \divergence (u\eta^2)=u\cdot D \eta^2
\end{array}
\end{equation*}
s.t.
\begin{equation*}\label{4.3}
\renewcommand{\arraystretch}{1.25}
\begin{array}{lll}
 \text{spt}\varpi \subset  Q_{\frac{3}{2}r}(x)
\end{array}
\end{equation*}
and for any $q>1$, the following estimate holds
\begin{equation}\label{4.4}
\renewcommand{\arraystretch}{1.25}
\begin{array}{lll}
 \|\varpi\|_{W^{1,q}}\leq C(q)\|u\cdot D\eta^2\|_{L^q}.
\end{array}
\end{equation}

Taking the test function $\varphi=u\eta^2-\varpi$ in\eqref{a} we obtain
\begin{equation*}\label{4.5}
\renewcommand{\arraystretch}{1.25}
\begin{array}{lll}
\displaystyle\int_{\mathbb{R}^2} & DH(\varepsilon(u)):\varepsilon(u\eta^2)dx-\displaystyle\int_{\mathbb{R}^2}DH(\varepsilon(u)):\varepsilon(\varpi)dx
-\displaystyle\int_{\mathbb{R}^2}u^iu^j\\
&\partial_{i}(u^j\eta^2)dx+\displaystyle\int_{\mathbb{R}^2}u^iu^j\partial_{i}\varpi^jdx=0.
\end{array}
\end{equation*}
Hence,
\begin{equation}\label{4.6}
\renewcommand{\arraystretch}{1.25}
\begin{array}{lll}
\displaystyle\int_{\mathbb{R}^2}  DH(\varepsilon(u)):\varepsilon(u)\eta^2dx&=-\displaystyle\int_{\mathbb{R}^2} DH(\varepsilon(u)):u\otimes D\eta^2dx
+\displaystyle\int_{\mathbb{R}^2}DH(\varepsilon(u))\\
&\ \ \ :\varepsilon(\varpi)dx
+\displaystyle\int_{\mathbb{R}^2}u^iu^j\partial_{i}(u^j\eta^2)dx
-\displaystyle\int_{\mathbb{R}^2}u^iu^j\partial_{i}\varpi^jdx\\
&=:I+II+III+IV.
\end{array}
\end{equation}

Recalling the definition of $H$ and $\eta$ and applying Lemma\ref{L2.4} and Young's inequality,  we have for any $0<\delta<1$
\begin{equation}\label{4.7}
\renewcommand{\arraystretch}{1.25}
\begin{array}{lll}
I&\leq\displaystyle\int_{\mathbb{R}^2}h'(|\varepsilon(u)|)|u||D\eta|dx\leq
C \displaystyle\int_{\mathbb{R}^2}\biggl(h(|\varepsilon(u)|)^{\frac{1}{\tau}}+|\varepsilon(u)|\biggr)|u||D\eta|dx \\
&\leq \delta \displaystyle\int_{Q_{\frac{3}{2}r}(x)}h(|\varepsilon(u)|)dx+C(\tau, \delta)\biggl\{\dfrac{1}{r^2}
\displaystyle\int_{Q_{\frac{3}{2}r}(x)}|u|^2dx +\dfrac{1}{r^{\tau'}} \displaystyle\int_{Q_{\frac{3}{2}r}(x)}|u|^{\tau'}dx\biggr\}.
\end{array}
\end{equation}

We deal with $II$ in the same way and we obtain
\begin{equation*}\label{4.8}
\renewcommand{\arraystretch}{1.25}
\begin{array}{lll}
II\leq
 &\delta \displaystyle\int_{Q_{\frac{3}{2}r}(x)}h(|\varepsilon(u)|)dx+C(\tau, \delta)\biggl\{
 \displaystyle\int_{Q_{\frac{3}{2}r}(x)}|\varepsilon(\varpi)|^2dx \\
&+
  \displaystyle\int_{Q_{\frac{3}{2}r}(x)}|\varepsilon(\varpi)|^{\tau'}dx\biggr\}.
\end{array}
\end{equation*}
Thus it follows from\eqref{4.4} that 
\begin{equation}\label{4.9}
\renewcommand{\arraystretch}{1.25}
\begin{array}{lll}
II\leq
 \delta \displaystyle\int_{Q_{\frac{3}{2}r}(x)}h(|\varepsilon(u)|)dx+C(\tau, \delta)\biggl\{\dfrac{1}{r^2}
 \displaystyle\int_{Q_{\frac{3}{2}r}(x)}|u|^2dx +\dfrac{1}{r^{\tau'}} \displaystyle\int_{Q_{\frac{3}{2}r}(x)}|u|^{\tau'}dx\biggr\}.
\end{array}
\end{equation}

We estimate $III$ by integration by parts. We have by the equation $\divergence u=0$ and Young's inequality that for any $\varepsilon>0$,
\begin{equation*}\label{4.10}
\renewcommand{\arraystretch}{1.25}
\begin{array}{lll}
III&=\displaystyle\int_{\mathbb{R}^2}u^iu^j\partial_i(u^j\eta^2)dx=\displaystyle\int_{\mathbb{R}^2}u^iu^j\partial_iu^j\eta^2dx
+\displaystyle\int_{\mathbb{R}^2}u^iu^ju^j\partial_{i}\eta^2dx\\
&=\displaystyle\int_{\mathbb{R}^2}\frac{|u|^2}{2}u\cdot D\eta^2 \leq \dfrac{C}{r}\displaystyle\int_{Q_{\frac{3}{2}r}(x)}|u|^{3}dx \\
&\leq \varepsilon
\displaystyle\int_{Q_{\frac{3}{2}r}(x)}|u|^{4}dx+\dfrac{C}{\varepsilon r^2}\displaystyle\int_{Q_{\frac{3}{2}r}(x)}|u|^{2}dx.
\end{array}
\end{equation*}

It remains to deal with $IV$. By Young's inequality and\eqref{4.4}, we have
\begin{equation*}\label{4.11}
\renewcommand{\arraystretch}{1.25}
\begin{array}{lll}
IV\leq \varepsilon
\displaystyle\int_{Q_{\frac{3}{2}r}(x)}|u|^{4}dx+\dfrac{C}{\varepsilon r^2}\displaystyle\int_{Q_{\frac{3}{2}r}(x)}|u|^{2}dx.
\end{array}
\end{equation*}
Furthermore, Lemma\ref{L2.5} gives us 
\begin{equation*}\label{4.12}
\renewcommand{\arraystretch}{1.25}
\begin{array}{lll}
\displaystyle\int_{Q_{\frac{3}{2}r}(x)}|u|^{4}dx\leq C_0 \displaystyle\int_{Q_{\frac{3}{2}r}(x)}|u|^{2}dx
\displaystyle\int_{Q_{\frac{3}{2}r}(x)}|Du|^{2}dx+\dfrac{C_0}{r^2}(\displaystyle\int_{Q_{\frac{3}{2}r}(x)}|u|^{2}dx)^2.
\end{array}
\end{equation*}

We Choose $\varepsilon=\dfrac{\delta}{C_0(1+\int_{Q_{\frac{3}{2}r}(x)}|u|^{2}dx)}$. Then $III$ and $IV$ are both bounded from above by 
\begin{equation}\label{4.13}
\renewcommand{\arraystretch}{1.25}
\begin{array}{lll}
\delta \displaystyle\int_{Q_{\frac{3}{2}r}(x)}|Du|^{2}dx+C(\delta)\biggl\{\dfrac{1}{r^2} \int_{Q_{\frac{3}{2}r}(x)}|u|^{2}dx
+ \dfrac{1}{r^2} (\int_{Q_{\frac{3}{2}r}(x)}|u|^{2}dx)^2 \biggr\}.
\end{array}
\end{equation}

Now, we want to give an appropriate control for the term
$\int_{Q_{\frac{3}{2}r}(x)}|Du|^{2}dx$. Choosing the cut-off
function $\xi$ s.t. $\xi\in C_0^\infty(Q_{2r}(x))$, $0\leq \xi\leq
1$, $\xi=1$ on $Q_{\frac{3}{2}r}(x)$ and $|D\xi|\leq \frac{4}{r}$,  we deduce from Lemma\ref{L2.2} that
\begin{equation}\label{4.14}
\renewcommand{\arraystretch}{1.25}
\begin{array}{lll}
\displaystyle\int_{Q_{\frac{3}{2}r}(x)}|Du|^{2}dx&\leq \displaystyle\int_{Q_{2r}(x)}|D(u\xi)|^{2}dx \leq
C \displaystyle\int_{Q_{2r}(x)}|\varepsilon(u\xi)|^{2}dx\\
&\leq C \displaystyle\int_{Q_{2r}(x)}|\varepsilon(u)|^{2}dx +  \frac{C}{r^2} \displaystyle\int_{Q_{2r}(x)}|u|^{2}dx.
\end{array}
\end{equation}

Putting together the estimates\eqref{4.6},\eqref{4.7},\eqref{4.9},\eqref{4.13} and\eqref{4.14} and noting $ h(\varepsilon(u))\geq \frac{1}{2}h''(0)|\varepsilon(u)|^2, h''(0)>0$, we have
\begin{equation}\label{4.15}
\renewcommand{\arraystretch}{1.25}
\begin{array}{lll}
\displaystyle\int_{\mathbb{R}^2}  DH(\varepsilon(u)):\varepsilon(u)\eta^2dx \leq &C\delta \displaystyle\int_{Q_{2r}(x)}h(|\varepsilon(u)|)dx
+C(\delta, \tau)\biggl\{\dfrac{1}{r^{\tau'}} \displaystyle\int_{Q_{\frac{3}{2}r}(x)}|u|^{\tau'} \\
&dx+\dfrac{1}{r^2} \displaystyle\int_{Q_{2r}(x)}|u|^{2}dx+\dfrac{1}{r^2} \biggl(\displaystyle\int_{Q_{2r}(x)}|u|^{2}dx\biggr)^2 \biggr\}.
\end{array}
\end{equation}
Thus it follows from\eqref{3} that for any $\delta>0$ 
\[
\renewcommand{\arraystretch}{1.25}
\begin{array}{lll}
\displaystyle\int_{Q_{r}(x)} h(|\varepsilon(u)|) dx \leq & C\delta \displaystyle\int_{Q_{2r}(x)}h(|\varepsilon(u)|)dx
+C( \tau)\biggl\{\dfrac{1}{r^{\tau'}} \displaystyle\int_{Q_{2r}(x)}|u|^{\tau'}dx \\
&+\dfrac{1}{r^2} \displaystyle\int_{Q_{2r}(x)}|u|^{2}dx+\dfrac{1}{r^2} \biggl(\displaystyle\int_{Q_{2r}(x)}|u|^{2}dx\biggr)^2 \biggr\},
\end{array}
\]
from which,  by Lemma\ref{L2.6}, it follows that
\[
\renewcommand{\arraystretch}{1.25}
\begin{array}{lll}
 \displaystyle\int_{Q_R(x_0)}h(|\varepsilon(u)|)dx \leq  &C\biggl\{\dfrac{1}{R^{\tau'}} \displaystyle\int_{Q_{2R}(x_0)}|u|^{\tau'} dx
 +\dfrac{1}{R^2}\displaystyle\int_{Q_{2R}(x_0)}|u|^2 dx \\
&+\dfrac{1}{R^2}  \biggl(\displaystyle\int_{Q_{2R}(x_0)}|u|^2 dx \biggr)^2  \biggr\}.
\end{array}
\]
This finishes the proof of Lemma\ref{L4.1}.
\end{proof}

\begin{lemma}\label{L4.2}
Let $u\in C^1(\mathbb{R}^2, \mathbb{R}^2)$ be an entire weak solution of\eqref{a} and $2\leq \tau' < 4$. Then,
for any $x_0\in \mathbb{R}^2$, $R>0$  we have
\begin{equation}\label{4.16}
\renewcommand{\arraystretch}{1.25}
\begin{array}{lll}
 \displaystyle\int_{Q_R(x_0)}h(|\varepsilon(u)|)dx \leq  & C(\tau)\biggl\{\dfrac{1}{R^2}\displaystyle\int_{Q_{2R}(x_0)}|u|^2 dx
 +\dfrac{1}{R^{\bar\tau}}\displaystyle\int_{Q_{2R}(x_0)}|u|^2 dx\\
&+\dfrac{1}{R^2}  \biggl(\displaystyle\int_{Q_{2R}(x_0)}|u|^2 dx \biggr)^2 +\dfrac{1}{R^{ \bar\tau}}
 \biggl(\displaystyle\int_{Q_{2R}(x_0)}|u|^2 dx \biggr)^{\tau^{\star}+1}  \biggr\},
\end{array}
\end{equation}
where $\bar\tau=\frac{2\tau'}{4-\tau'}\geq 2$ and $\tau^{\star}=\frac{\tau'-2}{4-\tau'}$.
\end{lemma}

\begin{proof}
Returning to\eqref{4.15}, we
just need to
control $\frac{1}{r^{\tau'}}\int_{Q_{\frac{3}{2}r}(x)}|u|^{\tau'}$ $dx$. By
Young's inequality we have,  for any $\varepsilon>0$,
\begin{equation}\label{4.17}
\renewcommand{\arraystretch}{1.25}
\begin{array}{lll}
 \dfrac{1}{r^{\tau'}}\displaystyle\int_{Q_{\frac{3}{2}r}(x)}|u|^{\tau'} dx\leq \varepsilon \displaystyle\int_{Q_{\frac{3}{2}r}(x)}|u|^{4} dx
 +\displaystyle C(\tau)\dfrac{1}{\varepsilon^{\tau^{\star}}}\dfrac{1}{r^{\bar\tau }}\int_{Q_{\frac{3}{2}r}(x)}|u|^{2} dx,
\end{array}
\end{equation}
where $\bar\tau=\frac{2\tau'}{4-\tau'}$. Moreover, by Lemma\ref{L2.5} we find
\begin{equation}\label{4.18}
\renewcommand{\arraystretch}{1.25}
\begin{array}{lll}
 \displaystyle\int_{Q_{\frac{3}{2}r}(x)}|u|^{4} dx\leq \displaystyle C_0\int_{Q_{\frac{3}{2}r}(x)}|u|^{2} dx\int_{Q_{\frac{3}{2}r}(x)}|Du|^{2} dx
 +\displaystyle C_0 \dfrac{1}{r^2}\biggl(\int_{Q_{\frac{3}{2}r}(x)}|u|^{2} dx\biggr)^2.
\end{array}
\end{equation}

Thus, putting together the estimates\eqref{4.17},\eqref{4.18} we have
\begin{equation}\label{4.19}
\renewcommand{\arraystretch}{1.25}
\begin{array}{lll}
 \dfrac{1}{r^{\tau'}}\displaystyle\int_{Q_{\frac{3}{2}r}(x)}|u|^{\tau'} dx & \leq  \displaystyle \varepsilon C_0\int_{Q_{\frac{3}{2}r}(x)}|u|^{2} dx\int_{Q_{\frac{3}{2}r}(x)}|Du|^{2} dx\\
&+\displaystyle \varepsilon C_0 \dfrac{1}{r^2}\biggl(\int_{Q_{\frac{3}{2}r}(x)}|u|^{2} dx\biggr)^2
+\displaystyle C(\tau)\dfrac{1}{\varepsilon^{\tau^{\star}}}\dfrac{1}{r^{\bar\tau }}\int_{Q_{\frac{3}{2}r}(x)}|u|^{2} dx.
\end{array}
\end{equation}

Letting $\varepsilon=\dfrac{\delta}{C_0(1+\int_{Q_{\frac{3}{2}r}(x)}|u|^{2} dx)}$ and putting together the estimates\eqref{4.15},\eqref{4.19}, in view of  $\eta=1$ in $Q_r(x)$, we end up with
\begin{equation*}\label{4.20}
\renewcommand{\arraystretch}{1.25}
\renewcommand{\arraystretch}{1.25}
\begin{array}{lll}
 \displaystyle\int_{Q_r(x)}h(|\varepsilon(u)|)dx \leq & C\delta \displaystyle\int_{Q_{2r}(x)}h(|\varepsilon(u)|)dx+\displaystyle\delta\int_{Q_{\frac{3}{2}r}(x)}|Du|^{2} dx\\
 &+ C(\tau, \delta)\biggl\{\dfrac{1}{r^2}\displaystyle\int_{Q_{2r}(x)}|u|^2 dx
 +\dfrac{1}{r^{\bar\tau}}\displaystyle\int_{Q_{2r}(x)}|u|^2 dx\\
&+\dfrac{1}{r^2}  \biggl(\displaystyle\int_{Q_{2r}(x)}|u|^2 dx \biggr)^2
+\dfrac{1}{r^{ \bar\tau}}  \biggl(\displaystyle\int_{Q_{2r}(x)}|u|^2 dx \biggr)^{\tau^{\star}+1}  \biggr\}.
\end{array}
\end{equation*}

To prove \eqref{4.16}, we may repeat the steps after\eqref{4.13} in the proof of Lemma\ref{L4.1}. We omit the details.
\end{proof}
\begin{lemma}\label{L4.3}
Let $u\in C^1(\mathbb{R}^2, \mathbb{R}^2)$ be an entire weak solution of\eqref{a}, then the following  holds

$(a)$ If $2\leq \tau' \leq 3$ and $u\in L^p(\mathbb{R}^2, \mathbb{R}^2)$, $1<p<2$, then $u$ must be a zero vector.

$(b)$ If $3< \tau' <4$ and $u\in L^p(\mathbb{R}^2, \mathbb{R}^2)$, $\tau'-2<p<2$, then $u$ must be a zero vector.
\end{lemma}

\begin{proof}
For any $q>2$, by Lemma\ref{L2.8} we have
\begin{equation*}\label{4.21}
\renewcommand{\arraystretch}{1.25}
\begin{array}{lll}
\displaystyle\int_{Q_R(x_0)}|u|^q dx\leq C(q) \biggl\{R^2\biggl(\displaystyle\int_{Q_R(x_0)}|Du|^2dx \biggr)^{\frac{q}{2}} +\dfrac{R^2}{R^q}\biggl(\displaystyle\int_{Q_R(x_0)}|u|^2 dx\biggr)^{\frac{q}{2}} \biggr\},
\end{array}
\end{equation*}
from which, together with Lemma\ref{L2.2},  it follows that
\begin{equation}\label{4.22}
\renewcommand{\arraystretch}{1.25}
\begin{array}{lll}
\displaystyle\int_{Q_R(x_0)}|u|^q dx\leq C(q) \biggl\{R^2\biggl(\displaystyle\int_{Q_R(x_0)}|\varepsilon(u)|^2dx \biggr)^{\frac{q}{2}} +\dfrac{R^2}{R^q}\biggl(\displaystyle\int_{Q_R(x_0)}|u|^2 dx\biggr)^{\frac{q}{2}} \biggr\}.
\end{array}
\end{equation}

Combining\eqref{4.16} and\eqref{4.22} we obtain
\begin{equation}\label{4.23}
\renewcommand{\arraystretch}{1.25}
\begin{array}{lll}
\displaystyle\int_{Q_{R}(x_0)}|u|^qdx &\leq
 C(q)\biggl\{R^2\biggl(\displaystyle\int_{Q_R(x_0)}h(\varepsilon(u))dx \biggr)^{\frac{q}{2}} +\dfrac{R^2}{R^q}\biggl(\displaystyle\int_{Q_R(x_0)}|u|^2 dx\biggr)^{\frac{q}{2}} \biggr\}\\
 &\leq  C(\tau,q)\biggl\{\dfrac{R^2}{R^q}\biggl(\displaystyle\int_{Q_{2R}(x_0)}|u|^2 dx\biggr)^{\frac{q}{2}}
 +\dfrac{R^2}{R^{\bar\tau\frac{q}{2}}}\biggl(\displaystyle\int_{Q_{2R}(x_0)}|u|^2 dx\biggr)^{\frac{q}{2}}\\
&+\dfrac{R^2}{R^q}  \biggl(\displaystyle\int_{Q_{2R}(x_0)}|u|^2 dx \biggr)^q
+\dfrac{R^2}{R^{\bar\tau\frac{q}{2}}}  \biggl(\displaystyle\int_{Q_{2R}(x_0)}|u|^2 dx \biggr)^{\frac{q}{2}(\tau^{\star}+1)}  \biggr\}.
\end{array}
\end{equation}

On the other hand, for $1<p<2$,  H\"{o}lder's inequality gives
\begin{equation}\label{4.24}
\renewcommand{\arraystretch}{1.25}
\begin{array}{lll}
\displaystyle\int_{Q_{2R}(x_0)}|u|^2dx\leq \displaystyle\biggl(\int_{Q_{2R}(x_0)}|u|^pdx\biggr)^\frac{1}{p'}
\biggl(\int_{Q_{2R}(x_0)}|u|^qdx\biggr)^\frac{1}{q'},
\end{array}
\end{equation}
where $p'=\frac{q-p}{q-2}$, $q'=\frac{q-p}{2-p}$.

Putting together the estimates\eqref{4.23} and\eqref{4.24} we deduce

\begin{equation}\label{4.25}
\renewcommand{\arraystretch}{1.25}
\begin{array}{lll}
\displaystyle\int_{Q_{R}(x_0)}|u|^qdx &\leq \displaystyle C(\tau,q)\biggl\{\dfrac{R^2}{R^q}
\biggl(\int_{Q_{2R}(x_0)}|u|^pdx \biggr)^{\frac{q}{2p'}} \biggl(\int_{Q_{2R}(x_0)}|u|^qdx \biggr)^{\frac{q}{2q'}}\\
&+\displaystyle\dfrac{R^2}{R^{\bar{\tau}\frac{q}{2}}} \biggl(\int_{Q_{2R}(x_0)}|u|^pdx
\biggr)^{\frac{q}{2p'}} \biggl(\int_{Q_{2R}(x_0)}|u|^qdx \biggr)^{\frac{q}{2q'}} \\
&+ \displaystyle\dfrac{R^2}{R^q} \biggl(\int_{Q_{2R}(x_0)}|u|^pdx \biggr)^{\frac{q}{p'}}
\biggl(\int_{Q_{2R}(x_0)}|u|^qdx \biggr)^{\frac{q}{q'}}\\
&+\displaystyle\dfrac{R^2}{R^{\bar{\tau}\frac{q}{2}}} \biggl(\int_{Q_{2R}(x_0)}|u|^pdx \biggr)^{\frac{q(\tau^{\star}+1)}{2p'}}
 \biggl(\int_{Q_{2R}(x_0)}|u|^qdx \biggr)^{\frac{q(\tau^{\star}+1)}{2q'}} \biggr\}.
\end{array}
\end{equation}

Notice that we always have $2-p<1$ and $\frac{2-p}{4-\tau'}<1$, under our assumption $(a)$ or $(b)$ on $\tau'$ and $p$. Since  $\lim\limits_{q\rightarrow \infty}\frac{q}{q'}=2-p$  and
$\lim\limits_{q\rightarrow \infty}\frac{q(\tau^{\star}+1)}{2q'}=\frac{2-p}{4-\tau'} $, we can choose $q$ large enough s.t  $\frac{q}{q'}<1, \frac{q(\tau^{\star}+1)}{2q'}<1$.
Thus, for any $\delta>0$,  using Young inequality in\eqref{4.25} we have

\begin{equation*}\label{4.26}
\renewcommand{\arraystretch}{1.25}
\begin{array}{lll}
\displaystyle\int_{Q_{R}(x_0)}|u|^qdx &\leq \delta \displaystyle \int_{Q_{2R}(x_0)}|u|^qdx+ \displaystyle
C(\tau,\delta,q)\biggl\{\biggl(\dfrac{R^2}{R^q}\biggr)^{\alpha_1} \biggl(\int_{Q_{2R}(x_0)}|u|^pdx \biggr)^{\beta_1} \\
&+\displaystyle\biggl(\dfrac{R^2}{R^{\bar{\tau}\frac{q}{2}}}\biggr)^{\alpha_2} \biggl(\int_{Q_{2R}(x_0)}|u|^pdx \biggr)^{\beta_2}
+ \bigg(\displaystyle\dfrac{R^2}{R^q}\biggr)^{\alpha_3} \biggl(\int_{Q_{2R}(x_0)}|u|^pdx \biggr)^{\beta_3} \\
&+\displaystyle\biggl(\dfrac{R^2}{R^{\bar{\tau}\frac{q}{2}}}\biggr)^{\alpha_4} \biggl(\int_{Q_{2R}(x_0)}|u|^pdx \biggr)^{\beta_4} \biggr\},
\end{array}
\end{equation*}
where $\alpha_i,  \beta_i,  1\leq i \leq 4$, are positive numbers.

By Lemma\ref{L2.6} we obtain
\begin{equation}\label{4.27}
\renewcommand{\arraystretch}{1.25}
\begin{array}{lll}
\displaystyle\int_{Q_{R}(x_0)}|u|^qdx &\leq  \displaystyle C(\tau,q)\biggl\{\biggl(\dfrac{R^2}{R^q}\biggr)^{\alpha_1}
\biggl(\int_{Q_{2R}(x_0)}|u|^pdx \biggr)^{\beta_1} \\
&+\displaystyle\biggl(\dfrac{R^2}{R^{\bar{\tau}\frac{q}{2}}}\biggr)^{\alpha_2} \biggl(\int_{Q_{2R}(x_0)}|u|^pdx \biggr)^{\beta_2}
+ \bigg(\displaystyle\dfrac{R^2}{R^q}\biggr)^{\alpha_3} \biggl(\int_{Q_{2R}(x_0)}|u|^pdx \biggr)^{\beta_3} \\
&+\displaystyle\biggl(\dfrac{R^2}{R^{\bar{\tau}\frac{q}{2}}}\biggr)^{\alpha_4} \biggl(\int_{Q_{2R}(x_0)}|u|^pdx \biggr)^{\beta_4} \biggr\}.
\end{array}
\end{equation}

Letting $R\rightarrow \infty$ in\eqref{4.27} and observing $\int_{\mathbb{R}^2}|u|^pdx<\infty$ we deduce that
\[
\renewcommand{\arraystretch}{1.25}
\begin{array}{lll}
\displaystyle\int_{\mathbb{R}^2} |u|^q dx =0,
\end{array}
\]
therefore, $u= 0$, and the proof is complete.
\end{proof}

\begin{lemma}\label{L4.4}
Let $u\in C^1(\mathbb{R}^2, \mathbb{R}^2)$ be an entire weak solution of\eqref{a}, then the following results hold

$(a)$ If $2\leq \tau' < 4$ and $u\in L^p(\mathbb{R}^2, \mathbb{R}^2)$, $p\geq 2$, then $u$ must be a zero vector.

$(b)$ If $ \tau'\geq 4$ and $u\in L^{ \tau'}(\mathbb{R}^2, \mathbb{R}^2)$,  then $u$ must be a zero vector.

\end{lemma}

\begin{proof}
For  $(a)$,  we will first show that the integrals of $h(|\varepsilon(u)|)$ and $|Du|^2$ are both local uniformly bounded. In fact, for any $x_0\in \mathbb{R}^2$, choosing $R=2$ in\eqref{4.16} and recalling
the condition\eqref{1} and\eqref{2} we obtain
\begin{equation*}\label{4.28}
\renewcommand{\arraystretch}{1.25}
\begin{array}{lll}
\displaystyle\int_{Q_2(x_0)}|\varepsilon(u)|^2dx\leq C \displaystyle\int_{Q_2(x_0)}h(|\varepsilon(u)|)dx \leq C(\|u\|_{L^p}, \tau),
\forall x_0\in \mathbb{R}^2,
\end{array}
\end{equation*}
from which, together with Lemma\ref{L2.2}, it gives
\begin{equation*}\label{4.29}
\renewcommand{\arraystretch}{1.25}
\begin{array}{lll}
\displaystyle\int_{Q_2(x_0)}|Du|^2dx\leq  C(\|u\|_{L^p}, \tau), \forall x_0\in \mathbb{R}^2.
\end{array}
\end{equation*}

Now Lemma\ref{L3.1} gives us 
\begin{equation*}\label{4.30}
\renewcommand{\arraystretch}{1.25}
\begin{array}{lll}
\displaystyle\int_{Q_1(x_0)}Wdx&\leq  C(\|u\|_{L^p},\tau)+C\displaystyle\int_{Q_2(x_0)} |u| dx\\
&\leq C(\|u\|_{L^p},\tau).
\end{array}
\end{equation*}

Thus by the equality $|D^2u(x)|\leq C |D\varepsilon(u)(x)|$ we have 
\begin{equation*}\label{4.31}
\renewcommand{\arraystretch}{1.25}
\begin{array}{lll}
\displaystyle\int_{Q_1(x_0)}|D^2u|^2dx \leq \displaystyle C\int_{Q_1(x_0)}|D\varepsilon(u)|^2dx\leq \displaystyle
C\int_{Q_1(x_0)}Wdx\leq  C(\|u\|_{L^p},\tau).
\end{array}
\end{equation*}
Now by  Sobolev's imbedding theorem,  we know $u\in
L^{\infty}(\mathbb{R}^2, \mathbb{R}^2)$. Hence, by Lemma\ref{L2.7},  
$u$ must be a constant vector. Since $u\in L^p(\mathbb{R}^2,
\mathbb{R}^2)$, then $u= 0$.

The proof of $(a)$ is complete. For $(b)$, we prove in the same way. Since $u\in L^{ \tau'}(\mathbb{R}^2, \mathbb{R}^2)$, Lemma\ref{L4.1}
gives us the uniform estimate\eqref{4.28} in this case. The rest of the proof is exactly the same. This 
finishes the proof of Lemma\ref{4.4}.
\end{proof}

\begin{remark}
If $u\in C^1(\mathbb{R}^2,\mathbb{R}^2)$  is an entire weak solution of\eqref{a} and $\int_{\mathbb{R}^2}h(|\varepsilon(u)|)dx<\infty$, we can't deduce that $u$ must be a constant vector. The counter example is $u_1=-y, u_2=x$.
But if the integrals of $|u|^p, p\geq 1$ and $h(|\varepsilon(u)|)$ are both local uniformly bounded, then by Lemma\ref{L3.1} we know  $u\in L^\infty(\mathbb{R}^2,\mathbb{R}^2)$. Hence $u$ is a constant vector.
\end{remark}


\begin{acknowledgement}
The author was supported by the Academy of Finland. He thanks  Thomas Z\"urcher helped him with the \LaTeX.
\end{acknowledgement}

\bibliographystyle{alpha}
\bibliography{liouville2}

\begin{thebibliography}{MNRR96}

\bibitem[BFZ05]{BFZ}
M.~Bildhauer, M.~Fuchs, and X.~Zhong.
\newblock A lemma on the higher integrability of functions with applications to
  the regularity theory of two-dimensional generalized {N}ewtonian fluids.
\newblock {\em Manuscripta Math.}, 116(2):135--156, 2005.

\bibitem[Eva98]{Ev}
L.~C. Evans.
\newblock {\em Partial differential equations}, volume~19 of {\em Graduate
  Studies in Mathematics}.
\newblock American Mathematical Society, Providence, RI, 1998.

\bibitem[FS00]{FS}
M.~Fuchs and G.~Seregin.
\newblock {\em Variational methods for problems from plasticity theory and for
  generalized {N}ewtonian fluids}, volume 1749 of {\em Lecture Notes in
  Mathematics}.
\newblock Springer-Verlag, Berlin, 2000.

\bibitem[Fuc]{Fu1}
M.~Fuchs.
\newblock Liouville theorems for stationary flows of shear thickening fluids in
  the plane.
\newblock {\em J. Math. Fluid Mech.}
\newblock To appear.

\bibitem[Fuc12]{Fu2}
M.~Fuchs.
\newblock Stationary flows of shear thickening fluids in 2{D}.
\newblock {\em J. Math. Fluid Mech.}, 14(1):43--54, 2012.

\bibitem[FZ12]{FZ}
M.~Fuchs and G.~Zhang.
\newblock Liouville theorems for entire local minimizers of energies defined on
  the class ${L} \log {L}$ and for entire solutions of the stationary
  {P}randtl-{E}yring fluid model.
\newblock {\em Calc. Var. Partial Differential Equations}, 44(1-2):271--295,
  2012.

\bibitem[Gal94a]{Ga1}
G.~P. Galdi.
\newblock {\em An introduction to the mathematical theory of the
  {N}avier-{S}tokes equations. {V}ol. {I}}, volume~38 of {\em Springer Tracts
  in Natural Philosophy}.
\newblock Springer-Verlag, New York, 1994.
\newblock Linearized steady problems.

\bibitem[Gal94b]{Ga2}
G.~P. Galdi.
\newblock {\em An introduction to the mathematical theory of the
  {N}avier-{S}tokes equations. {V}ol. {II}}, volume~39 of {\em Springer Tracts
  in Natural Philosophy}.
\newblock Springer-Verlag, New York, 1994.
\newblock Nonlinear steady problems.

\bibitem[GM82]{GM}
M.~Giaquinta and G.~Modica.
\newblock Nonlinear systems of the type of the stationary {N}avier-{S}tokes
  system.
\newblock {\em J. Reine Angew. Math.}, 330:173--214, 1982.

\bibitem[GT83]{GT}
D.~Gilbarg and N.~S. Trudinger.
\newblock {\em Elliptic partial differential equations of second order}, volume
  224 of {\em Grundlehren der Mathematischen Wissenschaften [Fundamental
  Principles of Mathematical Sciences]}.
\newblock Springer-Verlag, Berlin, second edition, 1983.

\bibitem[GW78]{GW}
D.~Gilbarg and H.~F. Weinberger.
\newblock Asymptotic properties of steady plane solutions of the
  {N}avier-{S}tokes equations with bounded {D}irichlet integral.
\newblock {\em Ann. Scuola Norm. Sup. Pisa Cl. Sci. (4)}, 5(2):381--404, 1978.

\bibitem[KNSS09]{KNSS}
G.~Koch, N.~Nadirashvili, G.~A. Seregin, and V.~Sverak.
\newblock Liouville theorems for the {N}avier-{S}tokes equations and
  applications.
\newblock {\em Acta Math.}, 203(1):83--105, 2009.

\bibitem[Lad69]{La}
O.~A. Ladyzhenskaya.
\newblock {\em The mathematical theory of viscous incompressible flow}.
\newblock Second English edition, revised and enlarged. Translated from the
  Russian by Richard A. Silverman and John Chu. Mathematics and its
  Applications, Vol. 2. Gordon and Breach Science Publishers, New York, 1969.

\bibitem[MNRR96]{MNRR}
J.~Malek, J.~Necas, M.~Rokyta, and M.~Ruzicka.
\newblock {\em Weak and measure-valued solutions to evolutionary {PDE}s},
  volume~13 of {\em Applied Mathematics and Mathematical Computation}.
\newblock Chapman \& Hall, London, 1996.

\bibitem[Tem83]{Te}
R.~Temam.
\newblock {\em Mathematical problems in plasticity}, volume~12 of {\em
  M\'ethodes Math\'ematiques de l'Informatique [Mathematical Methods of
  Information Science]}.
\newblock Gauthier-Villars, Montrouge, 1983.

\bibitem[Tem84]{Te1}
R.~Temam.
\newblock {\em Navier-{S}tokes equations}, volume~2 of {\em Studies in
  Mathematics and its Applications}.
\newblock North-Holland Publishing Co., Amsterdam, third edition, 1984.
\newblock Theory and numerical analysis, With an appendix by F. Thomasset.

\bibitem[Zha]{Z}
G.~Zhang.
\newblock A note on {L}iouville theorems for stationary flows of shear
  thickening fluids in the plane.
\newblock Submitted.

\end{thebibliography}

\end{document}